\documentclass[12pt,reqno]{amsart}
\usepackage[a4paper,margin=2.5cm,top=2.5cm,bottom=2.5cm,centering,headheight=3ex,headsep=4ex,vcentering]{geometry}

\usepackage{amssymb,amsfonts,amsmath,amsthm,charter,color,xcolor}
\usepackage{orcidlink}

\theoremstyle{plain}

\numberwithin{equation}{section}

\usepackage{booktabs}

\usepackage{color}
\definecolor{ao(english)}{rgb}{0.0, 0.0, 1.0}
\hypersetup{colorlinks=true, linkcolor=ao(english),citecolor=ao(english)}

\usepackage[normalem]{ulem} 

\numberwithin{equation}{section}

\usepackage{hyperref}
\usepackage{color}

\newtheorem{theorem}{Theorem}[section]
\newtheorem{conjecture}[theorem]{Conjecture}

\newtheorem{lemma}[theorem]{Lemma}
\newtheorem{remark}[theorem]{Remark}
\newtheorem{defn}[theorem]{Definition}

\title[Arithmetic Properties of Generalized Cubic and Overcubic Partitions]{Arithmetic Properties of Generalized Cubic and Overcubic Partitions}

\author[H. Das]{Hirakjyoti Das\, \orcidlink{0000-0002-5540-5625}}
\address[H Das]{Department of Mathematics, B. Borooah College, Guwahati 781007, Assam, India}
\email{hirak@bborooahcollege.ac.in}
\author[S. Maity]{Saikat Maity}
\address[S. Maity \& M. P. Saikia]{Mathematical and Physical Sciences division, School of Arts and Sciences, Ahmedabad University, Ahmedabad 380009, Gujarat, India}
\email{saikat.maity@ahduni.edu.in \& manjil.saikia@ahduni.edu.in}
\author[M. P. Saikia]{Manjil P. Saikia\,\orcidlink{0000-0002-2997-6731}}

\thanks{Corresponding Author: Manjil P. Saikia (manjil.saikia@ahduni.edu.in)}



\allowdisplaybreaks
\begin{document}

\begin{abstract}
We prove several congruences satisfied by the generalized cubic and generalized overcubic partition functions, recently introduced by Amdeberhan, Sellers, and Singh. We also prove infinite families of congruences modulo powers of $2$ and modulo $12$ satisfied by the generalized overcubic partitions, as well as some density results that they satisfy. We use both elementary $q$-series techniques as well as the theory of modular forms to prove our results.
\end{abstract}

\maketitle

\section{Introduction}

An (integer) \textit{partition} of a natural number $n$ is a non-increasing sequence of integers $\lambda_1 \geq \lambda_2\geq \cdots \geq \lambda_k$, such that $\sum\limits_{i=1}^k\lambda_i=n$. We denote by $p(n)$, the number of partitions of $n$. For instance $p(5)=7$ and the seven partitions of $5$ are
\[
(5), \,(4,1), \, (3,2), \, (3,1,1), \, (2,2,1), \, (2,1,1,1), \, (1,1,1,1,1).
\]
It is well-known that their generating function is given by
\[
\sum_{n\geq 0}p(n)q^n=\prod_{n\geq 1}\frac{1}{(1-q^n)}=\frac{1}{(q;q)_\infty}=\frac{1}{f_1}
,\]
where we have used the shorthand notations
\[
(a;q)_\infty:=\prod_{n\geq 0}(1-aq^n), \quad f_k:=(q^k;q^k)_\infty.
\]

The arithmetic properties of partitions have been a fruitful source of mathematical discovery since their inception. Ramanujan gave three intriguing congruences for the partition function, namely he proved that, for all $n\geq 0$, we have
\[
p(5n+4)\equiv 0\pmod 5, \quad p(7n+5)\equiv 0 \pmod7, \quad  p(11n+6)\equiv 0 \pmod{11}.
\]
Since then, mathematicians have been on the lookout for such Ramanujan congruences for other restricted partition functions as well. This has led to a proliferation of research in the theory of partitions by giving rise to numerous generalizations of partitions.

One such generalization was recently studied by Amdeberhan, Sellers, and Singh \cite{AmdeberhanSellersSingh}. They studied generalized cubic partitions and generalized cubic overpartitions. A \textit{generalized cubic partition} of weight $n$ is a partition of $n\geq 1$ wherein each part may appear in $c\geq 1$ different colors. We denote by $a_c(n)$ the number of such generalized cubic partitions of $n$, and set $a_c(0):=1$ for all $c\geq 1$. The generating function for $a_c(n)$ is given by
\begin{equation}\label{gf:ac}
    \sum_{n\geq 0}a_c(n)q^n=\frac{1}{f_1f_2^{c-1}}.
\end{equation}
If we set $c=1$ then we obtain the ordinary integer partitions, and if we set $c=2$ then we obtain the \textit{cubic partitions} of $n$. The case $c=2$ was first studied by Chan \cite{Chan}. Following Chan's work, there has been a lot of follow-up work studying cubic partitions, in particular mathematicians such as Chan \cite{Chan2}, Xiong \cite{Xiong}, Sellers \cite{Sellers}, Gireesh \cite{Gireesh}, Yao \cite{Yao}, Hirschhorn \cite{Hirschhorn}, Merca \cite{Merca}, Baruah \& Das \cite{BaruahDas}, and Buragohain \& Saikia \cite{BuragohainSaikia} have studied arithmetic properties of cubic partitions over the years.

Amdeberhan, Sellers, and Singh \cite{AmdeberhanSellersSingh} also considered \textit{generalized overcubic partitions}. These are defined in the same way as generalized cubic partitions and denoted by $\bar a_c(n)$. The generating function for $\bar a_c(n)$ is given by
\begin{equation} \label{genfn_bar ac}
    \sum_{n\geq 0}\bar a_c(n)q^n=\frac{f_4^{c-1}}{f_1^2f_2^{2c-3}}.
\end{equation}
The case $c=1$ are the ordinary overpartitions first studied by Corteel and Lovejoy \cite{CorteelLovejoy}, while the case $c=2$ was first studied by Kim \cite{Kim}. Following the work of Kim, overcubic partitions and their generalizations have received a lot of attention. Mathematicians such as, Hirschhorn \cite{Hirschhorn2}, Lin \cite{Lin14, Lin17}, Sellers \cite{Sellers}, Lin--Shen--Wang \cite{LinShenWang}, Ray \& Barman \cite{RayBarman}, Nayaka--Dharmendra--Mahesh Kumar \cite{NayakaDharmendraMaheshKumar}, and Saikia \& Sarma \cite{SaikiaSarma} have studied their arithmetic properties, and it is an area of active research.

Amdeberhan, Sellers, and Singh \cite{AmdeberhanSellersSingh} obtained several congruences for the $a_c(n)$ and $\bar a_c(n)$ functions using both elementary and modular forms techniques. Subsequently, Guadalupe \cite{Guadalupe} obtained a few more congruences for the $a_c(n)$ function. The purpose of this paper is to extend their lists and prove several more congruences for both $a_c(n)$ and $\bar a_c(n)$. 

First, we prove the following isolated congruences for $a_c(n)$. 

\begin{theorem}\label{THM1}
    For all $n\geq 0$, we have 
    \begin{align}
        a_{37}(43n+12)&\equiv 0 \pmod{43},\label{37_Th1}\\  
        a_{41}(47n+21)&\equiv 0 \pmod{47},\label{41_Th1}\\ 
        a_{53}(59n+56)&\equiv 0 \pmod{59},\label{53_Th1}\\ 
        a_{61}(67n+19)&\equiv 0 \pmod{67},\label{61_Th1}\\ 
        a_{65}(71n+32)&\equiv 0 \pmod{71},\label{65_Th1}\\ 
        a_{73}(79n+62)&\equiv 0 \pmod{79},\label{73_Th1}\\ 
        a_{77}(83n+79)&\equiv 0 \pmod{83} \label{77_Th1}.
        \end{align}
\end{theorem} 

\begin{remark}
    After we finished the proof of Theorem \ref{THM1}, Guadalupe \cite{Guadalupe} updated his preprint on arXiv, which actually generalizes the above result. However, he uses elementary techniques to achieve this, so we keep our proof in this article.
\end{remark}

\begin{theorem} \label{THM2}
    For all $n\geq 0$, we have
    \begin{align}
    a_3(7^2n+39)&\equiv 0 \pmod{7^2}.\label{3 thm2}
    \end{align}
\end{theorem}

Theorems \ref{THM1} and \ref{THM2} are proved using the theory of modular forms in Section \ref{sec:isolated}.

We also give a complete characterization of the $\bar a_c(n)$ function modulo $4$ in the following results.

\begin{theorem}\label{bar a_c (n) equiv 4}
For all $n \geq 1$, we have
$$
\bar a_c(n) \equiv
\begin{cases} 
2 \pmod{4} & \text{if } n = k^2,  \, k \in \mathbb{Z}, \\ 
2(c+1) \pmod{4} & \text{if } n = 2k^2,  \, k \in \mathbb{Z}, \\ 
0 \pmod 4 & \text{otherwise}.
\end{cases}
$$
\end{theorem}

\begin{remark}
    Theorem \ref{bar a_c (n) equiv 4} generalizes a result of Sellers \cite[Theorem 2.5]{Sellers}.
\end{remark}

\noindent Theorem \ref{bar a_c (n) equiv 4} is proved in Section \ref{sec:character} using elementary techniques.

We now move on to prove a few infinite family of congruences.
\begin{theorem}\label{Thm a_c 8n+5/7}
    For all $n\geq 0$, $i\geq 1$ and $k\geq 3$, we have
    \begin{align}
   \label{Cong a_c 8n+5}\bar a_{2^ki-2^{k-1}-2}(8n+5)&\equiv 0 \pmod{2^{k+1}},\\
    \label{Cong a_c 8n+7}\bar a_{2^ki-2^{k-1}-2}(8n+7)&\equiv 0 \pmod{2^{k+2}}.
    \end{align}
\end{theorem} 

\begin{theorem}\label{thm:nn}
    For all $n\geq 0$, $k\geq 1$ and $i\geq 0$, we have
    \begin{align}
\label{thm -3 8n+7}\bar a_{2^ki+\frac{2^k-6}{2}}(8n+7)&\equiv 0 \pmod{2^{k+3}},\\
\label{thm -2 8n+7}\bar a_{2^ki+\frac{2^k-2}{2}}(8n+7)&\equiv 0 \pmod{2^{k+3}}.
    \end{align}
\end{theorem}

We also prove a family of congruences modulo $12$.
\begin{theorem}\label{thm:3cong}
    For all $n\geq 0$ and $i\geq 1$, we have
    \begin{align}
        \bar a_{9i}(9n+6)&\equiv 0 \pmod {12},\\
        \bar a_{9i+4}(9n+6)&\equiv 0 \pmod{12},\\
        \bar a_{9i+8}(9n+6)&\equiv 0 \pmod {12}.
    \end{align}
\end{theorem}

Theorems \ref{Thm a_c 8n+5/7}, \ref{thm:nn}, and \ref{thm:3cong} are proved in Section \ref{sec:infinite} using elementary techniques.

Along with the study of Ramanujan-type congruences, the study of distribution of the coefficients of a formal power series modulo $M$ is also an interesting problem to explore. Given an integral power series $A(q) := \displaystyle\sum_{n=0}^{\infty}a(n)q^n $ and $0 \leq r \leq M$, the arithmetic density $\delta_r(A,M;X)$ is defined as
\begin{equation*}
       \delta_r(A,M;X) = \frac{\#\{ n \leq X : a(n) \equiv r \pmod{M} \}}{X}.
  \end{equation*}
An integral power series $A$ is called \textit{lacunary modulo $M$} if
\begin{equation*}
   \lim_{X \to \infty} \delta_0(A,M;X)=1,
\end{equation*}
which means that almost all the coefficients of $A$ are divisible by $M$. It turns out that the partition function $\bar a_c(n)$ also satisfies such results. Specifically, we prove the following results in this article. 

\begin{theorem} \label{Lacunary modulo 2^k}
Let $k\geq1$, $\alpha \geq 1$ and $m\geq1$ be integers with $\gcd(m,2)=1$. Then the set $$\{n\in \mathbb{N}: \bar a_{2^\alpha m}(n) \equiv 0 \pmod {2^k}\}$$ has arithmetic density $1$, namely, 
$$\lim_{X\to\infty} \frac{\#\{0\leq n \leq X : \bar a_{2^\alpha m}(n) \equiv 0 \pmod {2^k}\}}{X} = 1.$$
\end{theorem}

\begin{theorem} \label{Lacunary modulo 3^k}
Let $k\geq1$, $\alpha \geq 1$ and $m\geq1$ be integers with $\gcd(m,2)=1$ such that $3^{k-2}\geq 2^{\alpha -3}m$, then the set $$\{n\in \mathbb{N}: \bar a_{2^\alpha m}(n) \equiv 0 \pmod {3^k}\}$$ has arithmetic density $1$, namely, 
$$\lim_{X\to\infty} \frac{\#\{0\leq n \leq X : \bar a_{2^\alpha m}(n) \equiv 0 \pmod {3^k}\}}{X} = 1.$$
\end{theorem}

The fact that the action of Hecke algebras on the spaces of modular forms of level $1$ modulo $2$ is locally nilpotent was first observed by Serre \cite{Serre2, Serre1} and proved by Tate \cite{Tate}. This was further generalized by Ono and Taguchi \cite{Ono-Taguchi2005} to higher levels. We can use Ono and Taguchi's result to prove the following theorem.

\begin{theorem} \label{last thm}
Let $n$ be a non-negative integer, $\alpha \geq 1$, $m\geq1$ are integers with $\gcd(m,2)=1$, and $N$ is a positive integer such that $N = 1, 3, 5, 15$ or $17$. Then there exists an integer $i\geq0$ such that for every $j\geq1$ and $q_1, q_2, \dots, q_{i+j}$ are odd primes not dividing $N$, we have 
$$\bar a_{2^{\alpha}m} \left(\frac{q_1 \cdots q_{i+j} \cdot n}{24}\right) \equiv 0 \pmod {3^j}.$$
\end{theorem}

Theorems \ref{Lacunary modulo 2^k}, \ref{Lacunary modulo 3^k}, and \ref{last thm} are proved in Section \ref{sec:density}.

This paper is arranged as follows: Section \ref{prelim} contains some preliminary theory required in our later proofs, Sections \ref{sec:isolated} -- \ref{sec:density} contains the proofs of the above results, and finally we close the paper with some concluding remarks in Section \ref{sec:conc}.

\section{Preliminaries}\label{prelim}

\subsection{Various results on modular forms}
Let $\mathbb{H}$ be the complex upper half-plane. We define the following matrix groups
\begin{align*}
\Gamma & :=\left\{\begin{bmatrix}
a  &  b \\
c  &  d      
\end{bmatrix}: a, b, c, d \in \mathbb{Z}, ad-bc=1
\right\},\\
\Gamma_{\infty} & :=\left\{
\begin{bmatrix}
1  &  n \\
0  &  1      
\end{bmatrix} \in \Gamma : n\in \mathbb{Z}  \right\}.
\end{align*}
For a positive integer $N$, let 
\begin{align*}
\Gamma_{0}(N) :=\left\{
\begin{bmatrix}
a  &  b \\
c  &  d      
\end{bmatrix} \in \Gamma : c\equiv~0\pmod N \right\}, \\
\Gamma_{1}(N) :=\left\{
\begin{bmatrix}
a  &  b \\
c  &  d      
\end{bmatrix} \in \Gamma_{0} (N) : a\equiv d\equiv~1\pmod N \right\}.
\end{align*} 

We denote by $M_{\ell}(\Gamma_1(N))$ for a positive integer $l$, the complex vector space of modular forms of weight $\ell$ with respect to $\Gamma_1(N)$. The index of $\Gamma_0(N)$ in $\Gamma$ is
\begin{align*}
 \left[\Gamma : \Gamma_0(N)\right] = N\prod_{\ell|N}\left(1+\frac{1}{\ell}\right), 
\end{align*}
where $\ell$ is a prime divisor of $N$.

\begin{defn}
If $\chi$ is a Dirichlet character modulo $N$, then a form $f(z) \in M_l(\Gamma_1(N))$ has Nebentypus character $\chi$ if 
\begin{align*}
    f\left(\frac{az+b}{cz+d}\right) = \chi(d) (cz+d)^\ell f(z), 
\end{align*}
for all $z \in \mathbb{H}$ and all 
$\begin{bmatrix}
a  &  b \\
c  &  d      
\end{bmatrix} \in \Gamma_{0} (N)$. The space of such modular forms is denoted by $M_{\ell}(\Gamma_0(N), \chi)$. 
\end{defn}

Recall that the Dedekind's eta-function $\eta(z)$ is defined by 
$$\eta(z) := q^{1/24} (q;q)_{\infty} = q^{1/24} \prod_{n=1}^{\infty} (1-q^n),$$ 
where $q := e^{2\pi iz}$ and $z \in \mathbb{H}$. A function $f(z)$ is called an eta-quotient if it is of the form 
$$f(z) = \prod_{\delta | N} \eta(\delta z)^{r_{\delta}},$$
where $N$ is a positive integer and $r_{\delta}$ is an integer. We now recall two theorems from \cite[p. 18]{ono2004web} which will be used to prove some of our results. 

\begin{theorem}\label{thm_ono1} \cite[Theorems 1.64 and 1.65]{ono2004web}
    If $f(z) = \prod_{\delta | N} \eta(\delta z)^{r_{\delta}}$ is an eta-quotient with $\ell = \frac{1}{2} \sum_{\delta | N} r_{\delta} \in \mathbb{Z}$, with 
\[
    \sum_{\delta\mid N} \delta r_{\delta}\equiv 0 \pmod{24}
\quad
\text{and} \quad
    \sum_{\delta\mid N} \frac{N}{\delta}r_{\delta}\equiv 0 \pmod{24}, 
\]
then $f(z)$ satisfies 
\begin{align*}
   f\left(\frac{az+b}{cz+d}\right)=\chi(d)(cz+d)^{\ell}f(z),
\end{align*}
for every 
$\begin{bmatrix}
	a  &  b \\
	c  &  d      
\end{bmatrix} \in \Gamma_0(N)$. 
Here the character $\chi$ is defined by $\chi(d):=\left(\frac{(-1)^{\ell} \prod_{\delta | N} \delta^{r_\delta}}{d}\right)$. 
In addition, if $c$, $d$, and $N$ are positive integers with $d|N$ and $\gcd(c,d)=1$, then the order of vanishing of $f(z)$ at the cusp $\dfrac{c}{d}$ is $\dfrac{N}{24} \sum_{\delta | N} \dfrac{\gcd(d, \delta)^2 r_\delta}{\gcd(d, \frac{N}{d}) d \delta}$. 
\end{theorem}

We also recall the following result of Sturm \cite{Sturm} which will be used in the next section.
\begin{theorem}\label{Sturm}
Let $k$ be an integer and $g(z)=\sum_{n=0}^\infty a(n)q^n$ a modular form of weight $k$ for $\Gamma_{0}(N)$. For any given positive integer $m$, if $a(n)\equiv 0 \pmod{m}$ holds for all 
$n\leq \dfrac{kN}{12} \prod\limits_{p~ \text{prime},~p\,|\,N}~\left(1+\dfrac{1}{p}\right)\,$
 then $a(n)\equiv 0 \pmod m$ holds for any $n\geq 0$.
\end{theorem}

We now recall the definition of Hecke operators.
Let $m$ be a positive integer and $f(z) = \sum_{n=0}^{\infty} a(n)q^n \in M_{\ell}(\Gamma_0(N),\chi)$. Then the action of Hecke operator $T_m$ on $f(z)$ is defined by 
\begin{align*}
f(z)\,|\,T_m := \sum_{n=0}^{\infty} \left(\sum_{d\mid \gcd(n,m)}\chi(d)d^{\ell-1}a\left(\frac{nm}{d^2}\right)\right)q^n.
\end{align*}
In particular, if $m=p$ is prime, we have 
\begin{align}\label{Tp}
f(z)\,|\,T_p := \sum_{n=0}^{\infty} \left(a(pn)+\chi(p)p^{\ell-1}a\left(\frac{n}{p}\right)\right)q^n.
\end{align}
We take by convention that $a(n/p)=0$ whenever $p \nmid n$. The next result  follows directly from \eqref{Tp}.

\begin{theorem} \label{theorem_hecke}
Let $p$ be a prime, $g(z)\in \mathbb{Z}[[q]], h(z)\in \mathbb{Z}[[q^p]]$, and $k > 1$. Then, we have
\begin{align*}
\left(g(z)h(z)\right)\,|\,T_p\equiv\left(g(z)|T_p\cdot h(z/p)\right)\pmod{p}.
\end{align*}	
\end{theorem}

We also need the following result in Section \ref{sec:density}.
\begin{theorem} \label{thm_ono2} \cite[pp. 43]{ono2004web}
Let $g(z) \in M_\ell(\Gamma_0(N), \chi)$ has a Fourier expansion 
$$g(z)=\sum_{n=0}^{\infty}b(n)q^n \in \mathbb{Z}[[q]],$$
Then for a positive integer $r$, there is a constant $\alpha>0$  such that
$$ \# \left\{0<n\leq X: b(n)\not\equiv 0 \pmod{r} \right\}= \mathcal{O}\left(\frac{X}{(\log{}X)^{\alpha}}\right).$$
Equivalently $$\lim_{X \to \infty} \frac{\# \left\{0<n\leq X: b(n)\not\equiv 0 \pmod{r} \right\}}{X} = 0.$$
\end{theorem}

\subsection{Radu's Algorithm} We also need some results which describe an algorithmic approach to proving partition concurrences, developed by Radu \cite{Radu2009}. For integers $M \geq 1$, suppose that $R(M)$ be the set of all the integer sequences $$(r_\delta) := (r_{\delta_1}, r_{\delta_2}, r_{\delta_3}, \ldots, r_{\delta_k})$$ indexed by all the positive divisors $\delta$ of $M$, where $1=\delta_1<\delta_2< \cdots <\delta_k=M$. 
For positive integers $m$, and any integer $s$, let $[s]_m$ denote the residue class of $s$ in $\mathbb{Z}_m:= \mathbb{Z}/ {m\mathbb{Z}}$. 
Let $\mathbb{Z}_m^{*}$ be the set of all invertible elements in $\mathbb{Z}_m$. Let $\mathbb{S}_m\subseteq\mathbb{Z}_m$  be the set of all squares in $\mathbb{Z}_m^{*}$. For $(r_\delta) \in R(M)$ and $t\in\{0, 1, \ldots, m-1\}$, we define a subset $P(t)\subseteq\{0, 1, \ldots, m-1\}$ by
\begin{align} \label{formulae for P(t)}
P(t):=&\left\{t' \in \{0, 1, \ldots, m-1\} : t'\equiv ts+\frac{s-1}{24}\sum_{\delta|M}\delta r_\delta \pmod{m},\right.\nonumber \\ &\quad \left.\text{ for some } [s]_{24m}\in \mathbb{S}_{24m}  \right\}.
\end{align}

Let $m, M, N$ be positive integers. For $\gamma=
\begin{bmatrix}
a  &  b \\
c  &  d     
\end{bmatrix} \in \Gamma$, $(r_\delta)\in R(M)$ and $(r'_\delta)\in R(N)$, we also define 
\begin{align*}
p(\gamma):=\min_{\lambda\in\{0, 1, \ldots, m-1\}}\frac{1}{24}\sum_{\delta|M}r_{\delta}\frac{\gcd(\delta (a+ k\lambda c), mc)^2}{\delta m}
\end{align*}
and 
\begin{align*}
p'(\gamma):=\frac{1}{24}\sum_{\delta|N}r'_{\delta}\frac{\gcd(\delta, c)^2}{\delta}.
\end{align*}

Suppose $m, M$ and $N$ are positive integers, $t\in \{0, 1, \ldots, m-1\}$ and $(r_{\delta})\in R(M)$. Let $k:=\gcd(m^2-1,24)$ and define the set $\Delta^{*}$ consists of all tuples $(m, M, N, t, (r_{\delta}))$ such that all of the following conditions are satisfied:
\begin{enumerate}
\item Prime divisors of $m$ are also prime divisors of $N$.
\item If $\delta|M$, then $\delta|mN$ for every $\delta\geq1$ with $r_{\delta} \neq 0$.
\item $24 | kN \sum_{\delta|M} \dfrac{r_{\delta} mN}{\delta}$.
\item $8 | kN \sum_{\delta | M} r_{\delta}$.  
\item  $\dfrac{24m}{\gcd{(-24kt-k{\sum_{{\delta}|M}}{\delta r_{\delta}}} , 24m)} | N$.
\item If $2|m$, then either $4|kN$ and $8|sN$ or $2|s$ and $8|(1-j)N$, where $\prod_{\delta|M}\delta^{|r_{\delta}|}=2^s\cdot j$, here $s$ and $j$  are nonnegative integers with $j$ odd.
\end{enumerate} 

We have the following two results useful for us in the next section.
\begin{lemma} \cite[Lemma 4.5]{Radu2009}\label{lemma (Radu2009)} 
Suppose that $(m, M, N, t, (r_{\delta}))\in\Delta^{*}$, $(r'_{\delta}) :=(r'_{\delta})_{\delta|N}\in R(N)$, $\{\gamma_1,\gamma_2, \ldots, \gamma_n\}\subseteq \Gamma$ is a complete set of representatives of the double cosets of $\Gamma_{0}(N) \backslash \Gamma/ \Gamma_\infty$, and $t_{min}=\min_{t' \in P(t)} t'$, 
\begin{align*}
\nu:= \frac{1}{24}\left\{ \left( \sum_{\delta|M}r_{\delta}+\sum_{\delta|N}r'_{\delta}\right)[\Gamma:\Gamma_{0}(N)] -\sum_{\delta|N} \delta r'_{\delta}-\frac{1}{m}\sum_{\delta|M}\delta r_{\delta}\right\} - \frac{ t_{min}}{m},
\end{align*} 
$p(\gamma_j)+p'(\gamma_j) \geq 0$ for all $1 \leq j \leq n$, and $\sum_{n=0}^{\infty} A(n) q^n := \prod_{\delta | M} f_\delta ^{r_\delta}$. If for some integers $u \geq 1$, all $t' \in P(t)$, and $0\leq n\leq \lfloor\nu\rfloor$, $A(mn+t') \equiv 0 \pmod u$ is true, then for integers $n \geq 0$ and all $t' \in P(t)$, we have $A(mn+t') \equiv 0 \pmod u$. 
\end{lemma}
\begin{lemma}\cite[Lemma 2.6]{Radu-Sellers2011} \label{lemma (Radu-Sellers2011)}
If $N$ or $\frac{1}{2}N$ is a square-free integer, then
		\begin{align*}
		\bigcup_{\delta|N}\Gamma_0(N)\begin{bmatrix}
		1  &  0 \\
		\delta  &  1      
		\end{bmatrix}\Gamma_ {\infty}=\Gamma.
		\end{align*}
\end{lemma}

\section{Proof of Theorems \ref{THM1} and \ref{THM2}}\label{sec:isolated}

\begin{proof}[Proof of Theorem \ref{THM1}]
Since the proofs of the congruences are similar, we only prove \eqref{37_Th1} in detail. We consider the $\eta$-quotient 
$$G_1(z) := \frac{\eta(z)^{816}}{\eta(2z)^{36}}.$$  
Using Theorem \ref{thm_ono1}, we see that $G_1 (z)$ is a modular form of weight $390$, and level $4$, so $G_1 (z) \in {M}_{390} (\Gamma_0 (4))$. From the definition of $a_c (n)$, we have 
$$G_1 (z) = \frac{q^{31}}{f_1 f_2 ^{36}} \cdot f_1 ^{817} = \left(\sum_{n=0}^{\infty} a_{37} (n) q^{n+31}\right) \cdot f_1 ^ {817}.$$
Applying the Hecke operator $T_{43}$ on $G_1 (z)$, we calculate that 
$$G_1 (z) \mid T_{43} \equiv \left(\sum_{n=0}^{\infty} a_{37} (43n + 12) q^{n+1}\right) \cdot f_1 ^{817} \pmod {43}.$$ 
We know that $G_1 (z) \mid T_{43} \in {M}_{390} (\Gamma_0 (4))$, By Theorem \ref{Sturm}, the Sturm bound for this space of forms is $195$. Using Mathematica, we verify that the coefficients of $q$-expansion of $G_1 (z) \mid T_{43}$ up to the desired bound are congruent to $0$ $\pmod{43}$. Thus,  $G_1 (z) \mid T_{43} \equiv 0 \pmod {43}$, which proves \eqref{37_Th1}.

For the other congruences in Theorem \ref{THM1}, we mention the relevant data required to prove them in the table below.
\begin{align*}
    \begin{tabular}{c c c c c}
        \toprule
        Congruences  & $\eta$-quotient & Level  &  Weight & Sturm bound \\  
        \midrule
        \eqref{41_Th1} & $\dfrac{\eta(z)^{704}}{\eta(2z)^{40}}$ & 2 & 332 & 83 \\
        \midrule
        \eqref{53_Th1} & $\dfrac{\eta(z)^{176}}{\eta(2z)^{52}}$ & 4 & 62  & 31 \\
        \midrule
        \eqref{61_Th1} & $\dfrac{\eta(z)^{1272}}{\eta(2z)^{60}}$ & 4 & 606 & 303 \\
        \midrule
        \eqref{65_Th1} & $\dfrac{\eta(z)^{1064}}{\eta(2z)^{64}}$ & 2 & 500 & 125  \\ 
        \midrule
        \eqref{73_Th1} & $\dfrac{\eta(z)^{552}}{\eta(2z)^{72}}$ & 2 & 240 & 60  \\ 
        \midrule
        \eqref{77_Th1} & $\dfrac{\eta(z)^{248}}{\eta(2z)^{76}}$ & 4 & 86  & 43 \\
        \bottomrule
    \end{tabular}
\end{align*}
\end{proof}

\begin{proof}[Proof of Theorem \ref{THM2}]
We notice
\[
\sum_{n\geq 0}a_3(n)q^n=\frac{1}{f_1f_2^2} \equiv \frac{f_1^{48}}{f_2^2f_7^7} \pmod{49}.
\]
So, we choose $(m,M,N,t,(r_\delta))=(49,14,14,39,(48,-2,-7,0))$. 
It easy to verify that this choice of $(m, M, N, t, (r_{\delta}))$ satisfies the $\Delta^{*}$ conditions and from equation \eqref{formulae for P(t)} we see that $P(t)=\{39\}$. By Lemma \ref{lemma (Radu-Sellers2011)}, we know that 
$\left\{
\begin{bmatrix}
1  &  0 \\
\delta  &  1      
\end{bmatrix}:\delta|14 \right\}$ forms a complete set of double coset representatives of $\Gamma_{0}(N) \backslash \Gamma/ \Gamma_\infty$.
Let $\gamma_{\delta}=
\begin{bmatrix}
1  &  0 \\
\delta  &  1      
\end{bmatrix}$, and for the choice of $(r'_\delta ) = (18,0,0,0)$ we see that 
$$p(\gamma_{\delta})+p'(\gamma_{\delta}) \geq 0, \textrm{\ \ \ for all } \delta | N.$$ By Lemma \ref{lemma (Radu2009)}, we compute $\lfloor\nu\rfloor=57$. 
Using Mathematica we verify \eqref{3 thm2} for $n\leq 57$, which proves the result.
\end{proof}

\section{Proof of Theorem \ref{bar a_c (n) equiv 4}}\label{sec:character}

Let 
\[
\bar F_c(q)_:=\sum_{n\geq 0}\bar a_c(n)q^n.
\]
From \eqref{genfn_bar ac}, we can easily prove that 
\begin{equation} \label{functional equation F_c (q)}
\bar F_c(q) = \varphi(q)\varphi(q^2)^{c-1} \bar F_c(q^2)^2,
\end{equation}
where 
$$
\varphi(q): = \sum_{k = - \infty}^{\infty} q^{k^2} = \frac{f_2^5}{f_1^2 f_4^2}
$$ 
is Ramanujan's $\varphi$-function.  We can then iterate (\ref{functional equation F_c (q)}) to prove the next lemma.

\begin{lemma} \label{theorem for overline F_(c-1)(q)}
Let $c \geq 3$ be an integer. Then 
    $$\overline{F}_{c-1}(q) = \varphi(q)\prod_{i\geq 1}\varphi(q^{2^i})^{c\cdot 2^{i-1}}. $$ 
\end{lemma}

\begin{proof}
We have
\begin{align*}
\bar F_{c-1}(q) 
&= 
\frac{f_4 ^{c-2}}{f_1 ^2 f_2^{2c-5}} = \varphi(q)\varphi(q^2)^{c-2} \bar F_{c-1}(q^2)^2 \\
&= 
\varphi(q)\varphi(q^2)^{c-2}\left( \varphi(q^2)\varphi(q^4)^{c-2} \bar F_{c-1}(q^4)^2 \right)^2\\
&= 
\varphi(q)\varphi(q^2)^{c} \varphi(q^4)^{2(c-2)} \bar F_{c-1}(q^4)^4 \\
&= 
\varphi(q)\varphi(q^2)^{c} \varphi(q^4)^{2(c-2)}\left(  \varphi(q^4)\varphi(q^8)^{c-2} \bar F_{c-1}(q^8)^2 \right)^4 \\
&= 
\varphi(q)\varphi(q^2)^{c} \varphi(q^4)^{2c} \varphi(q^8)^{4(c-2)} \bar F_{c-1}(q^8)^8  \\
&\,~ \vdots
\end{align*} 
The result follows by continuing to iterate (\ref{functional equation F_c (q)}) indefinitely. 
\end{proof}

Using Lemma \ref{theorem for overline F_(c-1)(q)}, we now prove Theorem \ref{bar a_c (n) equiv 4}.

\begin{proof}[Proof of Theorem \ref{bar a_c (n) equiv 4}]
From the Binomial theorem, it is clear that $\varphi(q^{2^i})^{t\cdot2^{i-1}} \equiv 1 \pmod 4$ for each $i \geq 2$. Thus, 
\begin{align*}
\sum_{n\geq 0} \bar a_c(n)q^n &= \varphi(q) \prod_{i\geq 1} \varphi(q^{2^i})^{{(c+1)} 2^{i-1}}\\
&\equiv \varphi(q) \varphi(q^2)^{c+1} \pmod 4\\
&= \left(1 + 2 \sum_{n\geq1} q^{n^2} \right) \left(1 + 2 \sum_{n\geq1} q^{2n^2} \right)^{c+1}\\
&= \left(1 + 2 \sum_{n\geq1} q^{n^2} \right) \left(\sum_{k=0}^{c+1} \binom{c+1}{k} 2^k \left(\sum_{n\geq1} q^{2n^2}\right)^k \right)\\
&\equiv \left(1 + 2 \sum_{n\geq1} q^{n^2} \right) \left(1 + 2 (c+1) \sum_{n\geq1} q^{2n^2} \right) \pmod 4\\
&\equiv 1 + 2 \sum_{n\geq1} q^{n^2} + 2 (c+1) \sum_{n\geq1} q^{2n^2} \pmod 4.
\end{align*}
Hence, the result follows from the above identity.    
\end{proof}

\section{Proofs of Theorems \ref{Thm a_c 8n+5/7}, \ref{thm:nn}, and \ref{thm:3cong}}\label{sec:infinite}

We need the following $2$- and $3$-dissection formulas.
\begin{lemma}
\label{lemma_1overf1squared}
We have 
\begin{align}
\label{2d phi}\varphi(q)&=\varphi\left(q^4\right)+2q \psi\left(q^8\right),\\
   \label{2d f1^2}f_1^2&= \frac{f_2 f_8^5}{f_4^2 f_{16}^2}-2q \frac{f_2 f_{16}^2}{f_8},\\
   \label{2d f1^22}\frac{1}{f_1^2} &= \frac{f_8^5}{f_2^5f_{16}^2} + 2q\frac{f_4^2f_{16}^2}{f_2^5f_8},\\
   \label{3d phi}\varphi(q)&=\varphi\left(q^9\right)+2q \Omega\left(q^3\right),\\
    \label{3d f2^2/f1}\dfrac{f_2^2}{f_1}&=\frac{f_6 f_9^2}{f_3 f_{18}}+q\frac{f_{18}^2}{f_9},\\
   \label{3d f1/f2}\dfrac{f_1}{f_4} &=\frac{f_6 f_9 f_{18}}{f_{12}^3}-q \frac{f_3 f_{18}^4}{f_9^2 f_{12}^3}-q^2 \frac{f_6^2 f_9 f_{36}^3}{f_{12}^4 f_{18}^2},
\end{align}
where $\psi(q):=\dfrac{f_2^2}{f_1}$ and $\Omega(q):=\dfrac{f_2^2 f_3 f_{12}}{f_1 f_{4} f_{6}}$.
\end{lemma}

\begin{proof}
     \eqref{2d phi} and \eqref{2d f1^22} are the same identity in \cite[Eq. (1.9.4)]{Power}, \eqref{2d f1^2} can be obtained by replacing $q$ by $-q$ in \eqref{2d phi}, while \eqref{3d phi} and \eqref{3d f2^2/f1} are \cite[Eq. (26.1.1), (26.1.2)]{Power}. We prove \eqref{3d f1/f2} below. We need the following identities
     \begin{align}
        \label{2d f1/f3^3} \frac{f_6 f_{12}^2 f_{18}^2 f_{36}^2}{f_3} &= \frac{f_{18}^9}{f_9^3}+q^3 \frac{f_6^3 f_{36}^6}{f_3 f_{12}^2},\\
        \label{2d f3^3/f1} f_9^3 f_{12}&= \frac{f_3 f_{12}^4 f_{18}^2}{f_6^2 f_{36}}+q^3 f_3 f_{36}^3,\\
        \label{3d f1/f2^2} \frac{f_1}{f_2^2} &= \frac{f_3^2 f_9^3}{f_6^6}-q \frac{f_3^3 f_{18}^3}{f_6^7}+q^2 \frac{f_3^4 f_{18}^6}{f_6^8 f_9^3},\\
        \label{3d f1f4/f2} \frac{f_1f_4}{f_2} &= \frac{f_3 f_{12} f_{18}^5}{f_6^2 f_9^2 f_{36}^2} - q \frac{f_9 f_{36}}{f_{18}}.
     \end{align}
     \eqref{2d f3^3/f1}  and \eqref{3d f1/f2^2} are  \cite[Eqs. (22.7.5) and (37.2.3)]{Power} in disguise, respectively.  \eqref{2d f1/f3^3} and \eqref{3d f1f4/f2} can be obtained by replacing $q$ by $-q$ in \eqref{2d f3^3/f1} and \eqref{3d f2^2/f1}, respectively. With the help of \eqref{3d f1/f2^2} and \eqref{3d f1f4/f2}, we have
     \begin{align*}
         \dfrac{f_1}{f_4}&=\dfrac{f_2}{f_4^2}\cdot \dfrac{f_1f_4}{f_2}\\
         &= \frac{f_3 f_9}{ f_{12}^5 f_{18} f_{36}^2}\left(\frac{f_{18}^9}{f_9^3}+q^3 \frac{f_6^3 f_{36}^6}{f_3 f_{12}^2}\right)-q \frac{f_6^2 f_{18}^2 f_{36}}{f_9^2 f_{12}^7}\left(f_9^3 f_{12}-f_3 f_{36}^3 q^3\right)\\
         &\quad - q^2 \frac{f_3 f_6 f_9 f_{36}}{f_{12}^6 f_{18}^4}\left(\frac{f_{18}^9}{f_9^3}+q^3 \frac{f_6^3 f_{36}^6}{f_3 f_{12}^2}\right).
     \end{align*}
     Hence, \eqref{3d f1/f2} is evident after employing \eqref{2d f1/f3^3} and \eqref{2d f3^3/f1} in the above identity.
\end{proof}

We also need the following easily proved generating functions.
\begin{lemma}\label{bar a_c (2n) and bar a_c (2n+1)}
For all $c\geq 1$, we have 
\begin{align*}
    \sum_{n\geq 0} \bar a_c(2n)q^n &= \frac{f_2^{c-1}f_4^5}{f_1^{2c+2}f_8^2},\\
\sum_{n\geq 0} \bar a_c(2n+1)q^n &= 2 \frac{f_8^2}{f_4} \left(\frac{f_2}{f_1^2}\right)^{c+1}.
\end{align*}
\end{lemma}

\begin{proof}
We have  
\begin{align*}
\sum_{n\geq 0} \bar a_c(n)q^n 
&= \frac{f_4 ^{c-1}}{f_1^2 f_2 ^{2c-3}}= \frac{f_4 ^{c-1}}{f_2 ^{2c-3}} \left(\frac{f_8^5}{f_2^5 f_{16}^2} + 2q \frac{f_4^2 f_{16}^2}{f_2^5 f_8} \right) \textrm{\ \ \ thanks to Lemma \ref{lemma_1overf1squared} }.
\end{align*}
Now, extracting the terms involving even and odd powers of $q$ respectively from the above and then setting $q^2\rightarrow q$, we obtain the given generating functions in the lemma.
\end{proof}

\begin{proof}[Proof of Theorem \ref{Thm a_c 8n+5/7}]
    Using Lemma \ref{bar a_c (2n) and bar a_c (2n+1)}, we have
    \begin{align}
        \sum_{n\geq 0} \bar a_{2^ki-2^{k-1}-2}(2n+1)q^n&\ =2\dfrac{f_2^{2^ki-2^{k-1}-1}f_8^2}{f_1^{2^{k+1}i-2^{k}-2}f_4}\equiv 2\dfrac{f_1^{2^k+2}f_8^2}{f_2^{2^{k-1}+1}f_4}\pmod{2^{k+2}}\notag,
    \end{align}
       which, with the help of the 2-dissection \eqref{2d f1^2} of $f_1^2$, gives the following dissections
    \begin{align}
     \label{4n+1}\sum_{n\geq 0} \bar a_{2^ki-2^{k-1}-2}(4n+1)q^n&\equiv 2\sum_{r= 0}^{2^{k-2}}2^{2r} \binom{2^{k-1}+1}{2r}q^r \dfrac{f_4^{5\times 2^{k-1}-12 r+7}}{f_2^{2^k-4 r+3} f_8^{2^k-8 r+2}} \pmod{2^{k+2}},\\
       \label{4n+3} \sum_{n\geq 0} \bar a_{2^ki-2^{k-1}-2}(4n+3)q^n&\equiv -4\sum_{r= 0}^{2^{k-2}}2^{2r} \binom{2^{k-1}+1}{2r+1}q^r \dfrac{f_4^{5\times 2^{k-1}-12 r+1}}{f_2^{2^k-4 r+1}f_{8}^{2^k-8 r-2}} \pmod{2^{k+2}}.
    \end{align}
    By the properties of binomial coefficients, for $r\ge 1$, we have
    \begin{align*}
    2^{2r} \binom{2^{k-1}+1}{2r}&\equiv 0 \pmod{2^{k}}, &
     2^{2r} \binom{2^{k-1}+1}{2r+1}&\equiv 0 \pmod{2^k}.
    \end{align*}
    Therefore, \eqref{4n+1} and \eqref{4n+3} reduce to
    \begin{align*}
        \sum_{n\geq 0} \bar a_{2^ki-2^{k-1}-2}(4n+1)q^n&\equiv 2 \dfrac{f_4^{5\times 2^{k-1}+7}}{f_2^{2^k+3} f_8^{2^k+2}} \pmod{2^{k+1}},\\
       \sum_{n\geq 0} \bar a_{2^ki-2^{k-1}-2}(4n+3)q^n&\equiv -4 (2^{k-1}+1) \dfrac{f_4^{5\times 2^{k-1}+1}}{f_2^{2^k+1}f_{8}^{2^k-2}} \pmod{2^{k+2}}.
    \end{align*}
    Extracting the coefficients of the odd powers of $q$ from the above clearly proves Theorem \ref{Thm a_c 8n+5/7}.
\end{proof}

\begin{proof}[Proof of Theorem \ref{thm:nn}]
    Using Lemma \ref{bar a_c (2n) and bar a_c (2n+1)}, we have
    \begin{align}
        \label{a_c 2n+1 mod 2^{k+3}}\sum_{n\geq 0} \bar a_{2^ki+2^{k-1}-3}(2n+1)q^n&\equiv 2\dfrac{f_2^{2^ki+2^{k-1}-2}f_8^2}{f_1^{2^{k+1}i+2^{k}-4}f_4}=2\dfrac{\psi\left(q^4\right)}{\varphi\left(-q\right)^{2^k  i+2^{k-1}-2}}\pmod{2^{k+3}}\\
        &\equiv 2\dfrac{\varphi\left(-q\right)^{2^{k-1}(6i+7)+2}\psi\left(q^4\right)}{\varphi\left(-q\right)^{2^{k+2}  (i+1)}}\nonumber \\&\equiv 2\dfrac{\varphi\left(-q\right)^{2^{k-1}(6i+7)+2}\psi\left(q^4\right)}{\varphi\left(-q^2\right)^{2^{k+1}  (i+1)}}\pmod{2^{k+3}}.
    \end{align}
    Employing the 2-dissection of $\varphi(q)$ \eqref{2d phi} and then extracting the terms with odd powers of $q$, we obtain
    \begin{align}
        \label{Thm 4n+3}\sum_{n\geq 0} \bar a_{2^ki+2^{k-1}-3}(4n+3)q^n&\equiv -2^2\dfrac{\psi\left(q^2\right)}{\varphi\left(-q\right)^{2^{k+1}  (i+1)}} \sum_{r=0}^{2^{k-2}(6i+7)}2^{2r}\binom{2^{k-1}(6i+7)+2}{2r+1}q^r\notag\\
        &\quad \times \varphi\left(q^2\right)^{2^{k-1}(6i+7)-2r+1}\psi\left(q^4\right)^{2r+1}\pmod{2^{k+3}}.
    \end{align}
    By the properties of binomial coefficients, for $r\ge 1, i\ge 0$, we have
    \begin{align*}
        2^{2r}\binom{2^{k-1}(6i+7)+2}{2r+1}&\equiv 0 \pmod{2^{k+1}}.
    \end{align*}
    Therefore, \eqref{Thm 4n+3} reduces to
    \begin{align*}
        \sum_{n\geq 0} \bar a_{2^ki+2^{k-1}-3}(4n+3)q^n&\equiv -2^2(2^{k-1}(6i+7)+2)\\& \quad \times \dfrac{\varphi\left(q^2\right)^{2^{k-1}(6i+7)+1}\psi\left(q^2\right)\psi\left(q^4\right)}{\varphi\left(-q^2\right)^{2^{k}  (i+1)}} \pmod{2^{k+3}}.
    \end{align*}
    From the above identity, we evidently arrive at \eqref{thm -3 8n+7} by extracting out the odd powers of $q$. The proof of \eqref{thm -2 8n+7} is similar to that of \eqref{thm -3 8n+7}, so we skip it.
\end{proof}

\begin{proof}[Proof of Theorem \ref{thm:3cong}]
The proofs of the congruences in the theorem are similar. So, we prove only the first one.    Under modulo 3, we have
    \begin{align*}
        \sum_{n=0}^\infty \bar a_{9i}(n)q^n= \dfrac{f_4^{9i-1}}{f_1^2f_2^{18i-3}}\equiv \dfrac{f_{12}^{3i}}{f_3f_6^{6i-1}}\cdot \dfrac{f_1}{f_4} \pmod{3}
    \end{align*}
    Invoking \eqref{3d f1/f2} and then extracting the terms involving $q^{3n}$, we obtain
    \begin{align*}
        \sum_{n=0}^\infty \bar a_{9i}(3n)q^n&\equiv \frac{f_3 f_{6} f_2^{2-6i} f_{4}^{3i-3}}{f_1} \equiv \dfrac{f_3f_6f_{12}^{i-1}}{f_6^{2i}}\cdot \dfrac{f_2^2}{f_1} \pmod{3}.
    \end{align*}
    Since there is no term involving $q^{3n+2}$ in the 3-dissection of $f_2^2/f_1$, for all $n\ge0$, the above identity gives
    \begin{align}
        \label{9i mod 3} \bar a_{9i}(9n+6)\equiv 0 \pmod{3}.
    \end{align}
    Again, we have
    \begin{align*}
        \sum_{n=0}^\infty \bar a_{9i}(n)q^n&= \dfrac{f_4^{9i-1}}{f_1^2f_2^{18i-3}}= \dfrac{1}{\varphi\left(-q\right)\varphi\left(-q^2\right)^{9i-1}}\\
        &=\dfrac{\varphi\left(-\omega q\right)\varphi\left(-\omega^2 q\right)}{\varphi\left(-q\right)\varphi\left(-\omega q\right)\varphi\left(-\omega^2 q\right)}\left(\dfrac{\varphi\left(-\omega q^2\right)\varphi\left(-\omega^2 q^2\right)}{\varphi\left(-q^2\right)\varphi\left(-\omega q^2\right)\varphi\left(-\omega^2 q^2\right)}\right)^{9i-1},
    \end{align*}
    where $\omega\neq 1$ is a cube root of unity.
    Due to \eqref{3d phi}, the above identity is equivalent to
    \begin{align*}
        \sum_{n=0}^\infty \bar a_{9i}(n)q^n&=\dfrac{\varphi\left(-q^9\right)\varphi\left(-q^{18}\right)^{9i-1}}{\varphi\left(-q^3\right)^4\varphi\left(-q^{6}\right)^{36i-4}}\Bigg(\left(\varphi\left(-q^9\right)-2\omega q \Omega\left(-q^3\right)\right)\\ &\quad \times \left(\varphi\left(-q^9\right)-2\omega^2 q \Omega\left(-q^3\right)\right)
       \left(\varphi\left(-q^{18}\right)-2\omega q^2 \Omega\left(-q^6\right)\right)^{9i-1}\\ &\quad \times \left(\varphi\left(-q^{18}\right)-2\omega^2 q^2 \Omega\left(-q^6\right)\right)^{9i-1}\Bigg).
    \end{align*}
    Under modulo 4, we have
    \begin{align*}
        \sum_{n=0}^\infty \bar a_{9i}(n)q^n&\equiv \dfrac{\varphi\left(-q^9\right)\varphi\left(-q^{18}\right)^{9i-1}}{\varphi\left(-q^3\right)^4\varphi\left(-q^{6}\right)^{36i-4}}\Bigg(\left(\varphi\left(-q^9\right)^2+2 q \varphi\left(-q^9\right)\Omega\left(-q^3\right)\right)\\
        &\quad \times \left(\varphi\left(-q^{18}\right)^2+2q^2 \varphi\left(-q^{18}\right)\Omega\left(-q^6\right)\right)^{9i-1}\Bigg) \pmod{4}\\
        &\equiv \dfrac{\varphi\left(-q^9\right)\varphi\left(-q^{18}\right)^{9i-1}}{\varphi\left(-q^3\right)^4\varphi\left(-q^{6}\right)^{36i-4}}\Bigg(\left(\varphi\left(-q^9\right)^2+2 q \varphi\left(-q^9\right)\Omega\left(-q^3\right)\right)\\
        &\quad \times \left(\varphi\left(-q^{18}\right)^{18i-2}+2(9i-1)q^2 \varphi\left(-q^{18}\right)^{18i-3}\Omega\left(-q^6\right)\right)\Bigg) \pmod{4}.
    \end{align*}
    Extracting the terms involving $q^{3n}$ from the above identity, we find that
    \begin{align*}
         \sum_{n=0}^\infty \bar a_{9i}(3n)q^n&\equiv \dfrac{\varphi\left(-q^3\right)^3\varphi\left(-q^{6}\right)^{27i-3}}{\varphi\left(-q\right)^4\varphi\left(-q^{2}\right)^{36i-4}} \pmod{4}.
    \end{align*}
    Repeating the above process, we obtain
    \begin{align*}
         \sum_{n=0}^\infty \bar a_{9i}(3n)q^n&\equiv \dfrac{\varphi\left(-q^9\right)^{12}\varphi\left(-q^{18}\right)^{108i-12}}{\varphi\left(-q^3\right)^{16}\varphi\left(-q^{6}\right)^{144i-16}} \pmod{4},
    \end{align*}
    from which, it is clear that for all $n\ge0$,
    \begin{align}
        \label{9i mod 4} \bar a_{9i}(9n+6)\equiv 0 \pmod{4}.
    \end{align}
    The congruences \eqref{9i mod 3} and \eqref{9i mod 4} together imply  the first congruence of the theorem.
\end{proof}

\section{Proofs of Theorem \ref{Lacunary modulo 2^k}, \ref{Lacunary modulo 3^k}, and \ref{last thm}}\label{sec:density}

The proof of Theorem \ref{Lacunary modulo 2^k} is immediate from a result of  Cotron \textit{et. al} \cite{cotronetal}. We define
\begin{align}
    G(\tau) :=\frac{\eta(\delta_1\tau)^{r_1}\eta(\delta_2\tau)^{r_2}\cdots\eta(\delta_u\tau)^{r_u}}{\eta(\gamma_1\tau)^{s_1}\eta(\gamma_2\tau)^{s_2}\cdots\eta(\gamma_t\tau)^{s_t}}=q^{\tfrac{E_G}{24}}\sum\limits_{n=0}^{\infty}b(n)q^{n},\label{G(tau)}
\end{align}
where $r_i , s_i, \delta_i,$ and $\gamma_i$ are positive integers with $\delta_1,...,\delta_u, \gamma_1,...,\gamma_t$ distinct, $u, t \geq 0$, and 
$$E_G: = \sum \limits_{i=1}^u\delta_ir_i-\sum \limits_{i=1}^t\gamma_is_i.$$
The weight of $G(\tau)$ is given by
\begin{align*}
    \frac{1}{2}\left(\sum_{i=1}^{u}r_i-\sum_{i=1}^{t}s_i\right).
\end{align*}
Also define $\mathcal{D_G}:=\gcd(\delta_1, \delta_2,\ldots,\delta_u)$.
Then, we have the following result.
\begin{theorem}\cite[Theorem 1.1]{cotronetal}\label{thm:cotron}
    Suppose $G(\tau)$ is an eta-quotient of the form \eqref{G(tau)} with integer weight. If $p$ is a prime such that $p^a$ divides $\mathcal{D_G}$ and
    \begin{align}  p^a\geq\sqrt{\frac{\sum_{i=1}^{t}\gamma_is_i}{\sum_{i=1}^{u}\dfrac{r_i}{\delta_i}}},
    \end{align}
then $G(\tau)$ is lacunary modulo $p^j$ for any positive integer $j$. Moreover, there exists a positive constant $\alpha$, depending on $p$ and $j$, such that the number of integers $n \leq X$ with $p^j$ not dividing $b(n)$ is $O\left(\dfrac{X}{\log^{\alpha}X}\right)$.
\end{theorem}

\begin{proof}[Proof of Theorem \ref{Lacunary modulo 2^k}]
We recall
\begin{align*}
    \sum_{n\geq 0}\bar a_c(n)q^n=\frac{f_4^{c-1}}{f_1^2f_2^{2c-3}}=\frac{\eta^{c-1}(4z)}{\eta^2(z)\eta^{2c-3}(2z)}.
\end{align*}
Following the notations used in \eqref{G(tau)} and the paragraph succeeding it, we have \[\delta_1=4, \quad r_1=c-1, \quad \gamma_1=1, \quad \gamma_2=2,\quad s_1=2, \quad \text{and}\quad s_2=2c-3.\] Also $\mathcal{D_G}=4$ and the weight is $-c/2 \in \mathbb{Z}$ iff $c=2^\alpha m$ with $\alpha\geq 1$ and $\gcd (2,m)=1$. Next, we see that $2^2|4$ and
\begin{align*}
    2^2\geq\sqrt{\frac{2+(4c-6)}{\dfrac{c-1}{4}}}=\sqrt{16}=4.
\end{align*}
Choosing $p=2$ and $a=2$ in Theorem \ref{thm:cotron}, we complete the proof.
\end{proof}

The proof of Theorem \ref{Lacunary modulo 3^k} do not follow from Theorem \ref{thm:cotron}. We present it below.

\begin{proof}[Proof of Theorem \ref{Lacunary modulo 3^k}]
Putting $c=2^\alpha m$ in \eqref{genfn_bar ac}, we have 
\begin{align} \label{gf bar a_2^alpha m}
 \sum_{n\geq 0}\bar a_{2^\alpha m}(n)q^n=\frac{f_4^{2^\alpha m-1}}{f_1^2f_2^{2^{\alpha+1}m-3}}   .
\end{align}
We define $$A_{\alpha, m}(z) := \frac{\eta^3 (24z)}{\eta(72z)}.$$
By the binomial theorem, for any positive integers $m$, $k$, and prime $p$ we have $$f_m^{p^k} \equiv f_{mp}^{p^{k-1}} \pmod {p^k}.$$ 
Therefore, 
\begin{align} \label{1 mod 3^(k+1)}
A^{3^k}_{\alpha, m}(z) := \frac{\eta^{3^{k+1}} (24z)}{\eta^{3^k}(72z)} \equiv 1 \pmod {3^{k+1}}.
\end{align}
Next we define, 
\begin{align*}
B_{\alpha, m, k}(z) &= \frac{\eta^{2^\alpha m -1} (96z)}{\eta^2 (24z) \eta^{2^{\alpha+1}m-3}(48z)} A^{3^k}_{\alpha, m}(z)
= \frac{\eta^{3^{k+1}-2}(24z) \eta^{2^\alpha m -1} (96z)}{\eta^{2^{\alpha +1}m-3} (48z) \eta^{3^k}(72z)} .
\end{align*}
From \eqref{gf bar a_2^alpha m} and \eqref{1 mod 3^(k+1)}, we have 
\begin{align}
B_{\alpha, m, k}(z) &= \frac{\eta^{2^\alpha m -1} (96z)}{\eta^2 (24z) \eta^{2^{\alpha+1}m-3}(48z)} = \frac{f_{96}^{2^\alpha m-1}}{f_{24}^2f_{48}^{2^{\alpha+1}m-3}} 
\equiv \sum_{n=0}^{\infty} \bar a_{2^\alpha m}(n) q^{24n} \pmod {3^{k+1}} \label{mod 3^(k+1)}.
\end{align} 

Now, we will prove that $B_{\alpha, m, k}(z)$ is a modular form. From Theorem \ref{thm_ono1}, we find that the level of $B_{\alpha, m, k}(z)$ is $N=96M$, where $M$ is the smallest positive integer such that 
$$96M \left(\frac{3^{k+1}-2}{24}+\frac{3-2^{\alpha+1}m}{48}-\frac{3^k}{72}+\frac{2^\alpha m -1}{96}\right) \equiv 0 \pmod {24}$$
which implies 
$$3M \cdot \left(4\left(3^k-3^{k-2}\right)-2^\alpha m - 1\right) \equiv 0 \pmod{24}.$$
Therefore, $M=2^3$ and $N=2^8 3$.  The cusps of $\Gamma_0(2^8 3)$ are given by fractions $c/d$ where $d|2^8 3$ and $\gcd(c,d)=1$. By Theorem \ref{thm_ono1}, $B_{\alpha, m, k}(z)$ is holomorphic at a cusp $c/d$ if and only if 
\begin{multline*}
   \left(3^{k+1}-2\right)\frac{\gcd(d,24)^2}{24}-\left(2^{\alpha+1}m-3\right)\frac{\gcd(d,48)^2}{48}-3^k \frac{\gcd(d,72)^2}{72}\\+\left(2^{\alpha}m-1\right)\frac{\gcd(d,96)^2}{96} \geq 0. 
\end{multline*}
 Equivalently, $B_{\alpha, m, k}(z)$ is holomorphic at a cusp $c/d$ if and only if the following inequalities are satisfied for the different values of $d$ as mentioned in the table below (note that we have simplified the inequalities as far as possible).
 \begin{align*}
    \begin{tabular}{c | c}
        \toprule
        Value(s) of $d$ & Inequality \\  
        \midrule
       $1, 2, 4, 8, 24$ & $32\cdot 3^k-9\cdot 2^\alpha m-9 \geq 0$ \\
        \midrule
        $3, 6, 12$ & $3^k-\dfrac{9}{32}(2^\alpha m+1)\geq 0$ \\
        \midrule
    $16, 48$ & $8\cdot 3^k-9\cdot 2^\alpha m+9\geq 0$ \\
        \midrule
    $32, 64, 128, 256$ & $3^{k-2}\geq 0$ \\
        \midrule
    $96, 192, 384, 768$ & $3^k\geq 0$ \\
        \bottomrule
    \end{tabular}
\end{align*}

Since $k\geq 2$, the last two inequalities in the above list are already satisfied. Also we note that if the third inequality is satisfied than the first two are automatically satisfied. So, we only need to verify if
\[
8\cdot 3^k-9\cdot 2^\alpha m+9\geq 0.
\]
But this is true because we have $3^{k-2}\geq 2^{\alpha -3}m$ in our hypothesis. Therefore, $B_{\alpha, m, k}(z)$ is holomorphic at a cusp $c/d$. The weight of $B_{\alpha, m, k}(z)$ is $\ell = 3^k - 2^{\alpha -1}m$, which is a positive integer, and the associated character is given by 
$$\chi(\bullet)=\left(\frac{(-1)^\ell 2^{2 \cdot 3^{k+1}-3 \cdot 2^\alpha m + 1} 3^{2 \cdot 3^k-2^\alpha m}}{\bullet}\right).$$
Thus, $B_{\alpha, m, k}(z) \in M_{3^k - 2^{\alpha -1}m}\left(\Gamma_0(2^8 3), \chi(\bullet)\right)$. 

Therefore by Theorem \ref{thm_ono2}, the Fourier coefficients of $B_{\alpha, m, k}(z)$ are almost divisible by $3^k$. Due to \eqref{mod 3^(k+1)}, this holds for $\bar a_{2^\alpha m}$ also. This completes the proof. 
\end{proof}

We recall the following theorem of Ono and Taguchi \cite{Ono-Taguchi2005} before proving Theorem \ref{last thm}.
\begin{theorem}\label{thm_Ono-Taguchi} \cite[Theorem 1.3 (1)]{Ono-Taguchi2005}  Let $n$ be a non-negative integer and $k$ be a positive integer. Suppose that $\chi$ is a Dirichlet character with conductor $2^3 \cdot N$, where $N = 1, 3, 5, 15, 17$. Then there is an integer $c \geq 0$ such that for every $f(z) \in M_k (\Gamma_0(2^a N), \chi) \cap \mathbb{Z}[[q]]$ and every $t \geq 1$, 
$$f(z)|T_{p_1}|T_{p_2}|\cdots T_{p_{c+t}}\equiv 0 \pmod {3^t}$$
whenever $p_1, p_2, \dots , p_{c+t}$ are odd primes not dividing $N$.
\end{theorem}

\begin{proof}[Proof of Theorem \ref{last thm}]
From \eqref{mod 3^(k+1)}, we have 
$$B_{\alpha, m, k}\equiv \sum_{n=0}^{\infty} \bar a_{2^\alpha m}(n) q^{24n} \pmod {3^{k+1}}$$ 
which yields 
\begin{equation} \label{x}
B_{\alpha, m, k} := \sum_{n=0}^{\infty} \mathcal{B}_{\alpha, m, k} (n) q^n \equiv \sum_{n=0}^{\infty} \bar a_{2^\alpha m}\left(\frac{n}{24}\right) q^{n} \pmod {3^{k+1}}.
\end{equation}

Now, $B_{\alpha, m, k}(z) \in M_{3^k - 2^{\alpha -1}m}\left(\Gamma_0(2^8 3), \chi(\bullet)\right)$, where $\chi(\bullet)$ is the associated character. From Theorem \ref{thm_Ono-Taguchi}, we get that there is an integer $i\geq0$ such that for any $j\geq1$, 
$$B_{\alpha, m, k} (z)|T_{q_1}|T_{q_2}|\cdots T_{q_{i+j}}\equiv 0 \pmod {3^j}$$
whenever $q_1, q_2, \dots, q_{i+j}$ are odd primes not dividing $N$. From the definition of Hecke operators, we have that if $q_1, q_2, \dots, q_{i+j}$ are distinct primes and if $n$ is coprime to $q_1 \cdots q_{i+j}$, then 
\begin{equation} \label{y}
    \mathcal{B}_{\alpha, m, k} (q_1 \cdots q_{i+j} \cdot n) \equiv 0 \pmod{3^j}
\end{equation}
Combining \eqref{x} and \eqref{y}, we complete the proof. 
\end{proof}

\section{Concluding Remarks}\label{sec:conc}

\begin{enumerate}
    \item We have the following conjecture.
    \begin{conjecture}
For all $n\geq 0$, we have
    \begin{align}
            a_{25}(31^2n+644)&\equiv 0 \pmod{31^2},\label{25 thm2}\\
    a_{41}(47^2n+256)&\equiv 0 \pmod{47^2},\label{41 thm2}\\
    a_{61}(67^2n+555)&\equiv 0 \pmod{67^2}.\label{61 thm2}
    \end{align}
\end{conjecture}
If we use a procedure similar to the proof of Theorem \ref{THM2}, it would yield us the following data.
\begin{align*}
\begin{tabular}{c c c c}
 \toprule
 Congruences  & $(m,M,N,t,(r_\delta))$ and $(r_\delta^\prime)$ & $P(t)$  &  $\lfloor \nu\rfloor $ \\  \midrule
\eqref{25 thm2} & $(961,62,62,644,(960,-24,-31,0))$ & $\{644\}$ &  $3813$\\& and $(49,0,0,0)$ &   &  \\
 \midrule
 \eqref{41 thm2} & $(2209,94,94,256,(2208,-40,-47,0))$ & $\{256\}$ &  $13208$\\& and $(81,0,0,0)$ &   &  \\
  \midrule
 \eqref{61 thm2} & $(4489,134,134,555,(4488,-60,-67,0))$ & $\{555\}$ &  $38091$ \\& and $(121,0,0,0)$ &   &  \\
\bottomrule
\end{tabular}
\end{align*}
Since the bound $\nu$ is large for the computation to take a significant amount of time in a modest computer, we leave the above as an open problem.
\item We also have the following conjecture, which cannot be proved similarly like the proof of Theorem \ref{THM2}.
\begin{conjecture}
    For all $n\geq 0$, we have
    \begin{align*}
    a_{25}(31^2n+r)&\equiv 0 \pmod{31^2} \quad r\in\{41, 81, 585\},\\
    a_{37}(43^2n+r)&\equiv 0 \pmod{43^2} \quad r\in\{716, 987\},\\
    a_{53}(59^2n+100)&\equiv 0 \pmod{59^2},\label{t53}\\
    a_{65}(71^2n+r)&\equiv 0 \pmod{71^2} \quad r\in \{41, 47\},\\
    a_{77}(83^2n+157)&\equiv 0 \pmod{83^2}.
    \end{align*}
\end{conjecture}
\item It is desirable to have elementary proofs for Theorem \ref{THM2} and the above conjectures.
\item In this paper we have not explored moduli other than powers of $2$ in detail. It would be natural to explore what happens if we change the modulus to another prime. To this end, we make the following conjecture, whose proof should be similar to the proof of Theorem \ref{thm:3cong}.
\begin{conjecture}
    For all $n\geq 0$ and $i\geq 1$, we have
    \begin{align*}
        \bar a_{3i+2}(9n+2)&\equiv 0 \pmod {6},\\
    \bar a_{9i+5}(9n+3)&\equiv 0 \pmod{12},\\
\bar a_{9i+8}(9n+3)&\equiv 0 \pmod{12},\\
 \bar a_{3i+2}(9n+5)&\equiv 0 \pmod{6},\\
  \bar a_{3i+2}(9n+8)&\equiv 0 \pmod{6}.
    \end{align*}
\end{conjecture}
\noindent After this paper was posted on ArXiv, Ghoshal and Jana \cite{ghoshal2025combinatorialproofresultgeneralized} announced a proof of this conjecture.
\end{enumerate}

{\small \subsection*{Declarations} The authors make the following declaration.
\begin{enumerate}
\item The last two authors are supported by a Start-Up Grant from Ahmedabad University, India (Reference No. URBSASI24A5).
    \item The authors have no relevant financial or non-financial interests to disclose.
    \item There is no data associated with this manuscript.
    \item All authors contributed to the study conception and design. All authors read and approved the final manuscript.
\end{enumerate}}

\end{document}